\documentclass{amsart}
\usepackage{amsfonts,amsmath,amssymb}
\usepackage{graphics, graphicx,color}

\newtheorem{theorem}{Theorem}

\newtheorem{lemma}{Lemma}

\newtheorem{proposition}{Proposition}
\newtheorem{corollary}{Corollary}

\theoremstyle{definition}

\theoremstyle{remark}

\newtheorem{remark}{Remark}

\newcommand{\ind}{\operatorname{ind}}

\newcommand{\dd}{\mathfrak{D}}

\newcommand{\hh}{\mathcal{H}}
\newcommand{\mix}{\mathcal{M}}

\newcommand{\prop}{\mathcal{P}}
\newcommand{\qprop}{\mathcal{Q}}
\newcommand{\ch}{{\bf 1}}

\newcommand{\ii}{\mathcal{I}}

\newcommand{\Ker}{\mathfrak{K}}

\newcommand{\bs}{\backslash}

\newcommand{\Hom}{\operatorname{Hom}}
\newcommand{\diag}{\operatorname{diag}}

\newcommand{\zz}{\mathbb{Z}}

\newcommand{\cc}{\mathbb{C}}
\newcommand{\qq}{\mathbb{Q}}

\newcommand{\A}{\mathbb{A}}

\newcommand{\Ind}{\operatorname{Ind}}

\newcommand{\sprod}[2]{\left\langle#1,#2\right\rangle}

\address{Humboldt-Universit\"at zu Berlin, Mathematich-Naturwissenschaftliche Fakult\"at II,
Institut f\"ur Mathematik, Sitz: Rudower Chausee 25, D-10099 Berlin,
Germany.} \email{offen@mathematik.hu-berlin.de}

\address{Einstein Institute of Mathematics,
Edmond J. Safra Campus, Givat Ram, The Hebrew University of
Jerusalem, Jerusalem, 91904, Israel} \email{sayag@math.huji.ac.il}

\begin{document}

\begin{abstract}
We show the uniqueness and disjointness of Klyachko models for
$GL_n$ over a non-archimedean local field. This completes, in
particular, the study of Klyachko models on the unitary dual. Our
local results imply a global rigidity property for the discrete
automorphic spectrum.
\end{abstract}

\author{Omer Offen}
\author{Eitan Sayag}
\title{Uniqueness and disjointness of Klyachko models}
\date{\today}

\maketitle

\section{Introduction}\label{intro}
In this work we show that over a local non-archimedean field, the
mixed (symplectic-Whittaker) models introduced by Klyachko in
\cite{kly} are disjoint and that multiplicity one is satisfied. In
\cite{os} we showed, over a $p$-adic field (a finite extension of
$\qq_p$), the existence of Klyachko models for unitary
representations. The up shot is then that for every irreducible,
unitary representation of $GL_n$ over a $p$-adic field there is a
unique Klyachko model where it appears and it appears there with
multiplicity one.

To formulate the main result more precisely we introduce some
notation. Let $F$ be a non-archimedean local field. For a positive
integer $r$, denote by $U_r$ the subgroup of upper triangular
unipotent matrices in $GL_r$ and let
\[
    Sp_{2k}=\{g \in GL_{2k}:{}^t gJ_{2k} g=J_{2k}\}
\]
where
\[
 J_{2k}=
    \left(
\begin{array}{cc}
  0  & w_k   \\
  -w_k  & 0   \\
\end{array}
\right)
\]
and $w_k \in GL_k(F)$ is the matrix with $(i,j)^{th}$ entry equal to
$\delta_{i,n+1-j}.$ Whenever $n=r+2k$ we consider the subgroup
$H_{r,2k}$ of $GL_n$ defined by
\[
    H_{r,2k} =\{
\left(
\begin{array}{cc}
  u  & X   \\
  0  &  h  \\
\end{array}
\right): u \in U_r,\, X \in M_{r \times 2k},\,h \in Sp_{2k}\}.
\]
Let $\psi$ be a non trivial character of $F.$ For $u=(u_{i,j}) \in
U_r(F)$ we set
\begin{equation}\label{eq: gen char}
    \psi_r(u)=\psi(u_{1,2}+\cdots+u_{r-1, r}).
\end{equation}
Let $\psi_{r,2k}$ be the character of $H_{r,2k}(F)$ defined by
\begin{equation}\label{eq: char on H}
    \psi_{r,2k}
    \left(
        \begin{array}{cc}
        u  & X   \\
        0  &  h  \\
        \end{array}
    \right)=\psi_r(u).
%    ,\,\left(
%        \begin{array}{cc}
%        u  & X   \\
%        0  &  h  \\
%        \end{array}
%    \right) \in H_{r,2k}(F).
\end{equation}
When $n=r+2k$ the space
\[
\mix_{r,2k}=\mbox{Ind}_{H_{r,2k}(F)}^{GL_{n}(F)}(\psi_{r})
\]
is referred to as a Klyachko model and we say that a representation
$\pi$ of $GL_{n}(F)$ admits the Klyachko model $\mix_{r,2k}$ if
$\Hom_{GL_n(F)}(\pi,\mix_{r,2k})\ne 0.$ Here $\Ind$ denotes the
functor of non-compact smooth induction and representations of
$GL_n(F)$ are always assumed to be smooth. The main result of this
paper is the following.
\begin{theorem}\label{thm:uniq-disj}
Let $F$ be a non-archimedean local field and let $\pi$ be an
irreducible representation of $GL_n(F)$ then
\begin{equation}\label{eq: mult one}
\dim_\cc(\Hom_{GL_n(F)}(\pi,\mathop{\oplus}\limits_{k=0}^{[\frac
n2]}\mix_{n-2k,2k}))\le 1.
\end{equation}
\end{theorem}
Denote by
\[
m_\pi=\dim_\cc(\Hom_{GL_n(F)}(\pi,\mathop{\oplus}\limits_{k=0}^{[\frac
n2]}\mix_{n-2k,2k}))
\]
the multiplicity of $\pi$ in the direct sum of the Klyachko models.
When $F$ is a finite field, it is proved in \cite{MR1129515} that
$m_\pi=1$ for every irreducible representation $\pi$ of $GL_n(F).$
When $F$ is a non-archimedean local field it is shown in
\cite{MR1078382} that there exists an irreducible representation
$\pi$ of $GL_3(F)$ so that $m_\pi=0.$ Thus, we cannot expect in
general for the inequality \eqref{eq: mult one} to be an equality.
However, in \cite{os} we showed that if $F$ is a $p$-adic field then
$m_\pi \ge 1$ for every irreducible, unitary representation $\pi$ of
$GL_n(F).$ We therefore have the following.
\begin{corollary}\label{cor: unitary mult one}
Let $F$ be a $p$-adic field and let $\pi$ be an irreducible, unitary
representation of $GL_n(F)$ then
\[
\dim_\cc(\Hom_{GL_n(F)}(\pi,\mathop{\oplus}\limits_{k=0}^{[\frac
n2]}\mix_{n-2k,2k}))= 1.
\]
\end{corollary}
By Frobenius receiprocity \cite[\S2.28]{MR0425030} for a
representation $\pi$ of $GL_n(F)$ we have
\begin{equation}\label{eq: frob}
    \Hom_{GL_n(F)}(\pi,\mix_{r,2k})=\Hom_{H_{r,2k}(F)}(\pi,\psi_r).
\end{equation}
It follows that for an irreducible, unitary representation $\pi$
of $GL_n(F)$ there is a unique integer $0 \le \kappa(\pi)\le [\frac
n2]$ such that
\[
    \Hom_{H_{n-2\kappa(\pi),2\kappa(\pi)}(F)}(\pi,\psi_{n-2\kappa(\pi),2\kappa(\pi)})\ne
    0,
\]
i.e. such that $\pi$ is
$(H_{n-2\kappa(\pi),2\kappa(\pi)},\psi_{n-2\kappa(\pi),2\kappa(\pi)})$-\emph{distinguished}
and that the space of such functionals is one dimensional.
\begin{remark}
In \cite[Theorem 1]{mix}, when $F$ is a $p$-adic field and $n$ is
even, we exhibited a family of irreducible, unitary representations
of $GL_n(F)$ that are $Sp_n(F)$-distinguished. We promised in
\cite{mix} that in \cite{os} we will show that this family exhausts
all irreducible, unitary representations that are
$Sp_n(F)$-distinguished. Eventually, we postponed the delivery of
this statement to the current paper. It is immediate from Corollary
\ref{cor: unitary mult one}.
\end{remark}

In \cite{os} we also studied globally over a number field, the mixed
(symplectic-Whittaker) periods on the discrete automorphic spectrum
of $GL_n.$ Let $F$ be a number field and let $\psi$ be a non-trivial
character of $F \bs \A_F.$ We use \eqref{eq: gen char} to view
$\psi_r$ as a character of $U_r(\A_F)$ and \eqref{eq: char on H} to
view $\psi_{r,2k}$ as a character of $H_{r,2k}(\A_F).$ For an
automorphic form $\phi$ in the discrete spectrum automorphic
spectrum of $GL_n(\A_F)$ and a decomposition $n=r+2k$ we consider
the mixed period integral
\begin{equation}\label{period}
    P_{r,2k}(\phi)=\int_{H_{r,2k}(F)\bs
    H_{r,2k}(\A_F)}
    \phi(h)\,\psi_{r,2k}(h)\ dh.
\end{equation}
We say that an irreducible, discrete spectrum automorphic
representation $\pi$ of $GL_n(\A_F)$ is
$H_{r,2k}$-\emph{distinguished} if $P_{r,2k}$ is not identically
zero on the space of $\pi$. In \cite{os} we provided an explicit
integer $0 \le \kappa(\pi)\le [\frac n2]$ such that $\pi$ is
$(H_{n-2\kappa(\pi),2\kappa(\pi)},\psi_{n-2\kappa(\pi),2\kappa(\pi)})$-\emph{distinguished}.
Furthermore, we showed that this period integral is factorizable.
Corollary \ref{cor: unitary mult one} (particularly, the
disjointness of Klyachko models) then shows that there is a unique
such integer. Furthermore, it implies the following rigidity
property of the discrete automorphic spectrum of $GL_n.$
\begin{theorem}
Let $F$ be a number field and let $\pi=\otimes_v \pi_v$ be an
irreducible, discrete spectrum automorphic representation of
$G(\A_F)$. Then the following are equivalent:
\begin{enumerate}
\item {$\pi$ is $(H_{r,2k},\psi_r)-$distinguished;}

\item{$\pi_v$ is $(H_{r,2k},\psi_r)-$distinguished for all places
$v$ of $F$;}

\item{$\pi_{v_0}$ is $(H_{r,2k},\psi_r)-$distinguished for some finite
place $v_0$ of $F$.}
\end{enumerate}
\end{theorem}

The rest of this work is organized as follows. After setting up the
notation in \S\ref{notation}, in \S\ref{Reduction to Invariant
Distributions}-\S\ref{sec: Reduction to H-orbits} we reduce Theorem
\ref{thm:uniq-disj} to a statement about invariant distributions on
orbits. This statement is made more explicit in \S\ref{sec: explicit
property} and is then proved by induction in \S\ref{sec: induction}.

\section{Notation}\label{notation}

Let $F$ be a non-archimedean local field and for any positive
integer $r$ let $G_r=GL_r(F).$ We also set $G_0=\{1\}.$ Throughout,
we fix a positive integer $n$ and let $G=G_n.$ For a partition
$(n_1,\dots,n_t)$ of $n$ we denote by $P_{(n_1,\dots,n_t)}$ the
standard parabolic subgroup of $G$ of type $(n_1,\dots,n_t).$ It
consists of matrices in upper triangular block form. If
$P=P_{(n_1,\dots,n_t)}$ we denote by $\overline P$ the parabolic
opposite to $P.$ It consists of matrices in lower triangular form.
When we say that $P=MU$ is the standard Levi decomposition of $P$ we
mean that $U$ is its unipotent radical, and $M=P \cap \overline
P=\{\diag(g_1,\dots,g_t):g_i \in G_{n_i}\}.$ We then denote by
$\overline U$ the unipotent radical of $\overline P.$ We denote by
$a^{(r)}$ the $r$-tuple $(a,\dots,a)$, thus for example $P_{(1)^n}$
is the subgroup of upper triangular matrices in $G.$ For any
standard Levi subgroup $M$ of $G$ denote by $W_M$ the Weyl group of
$M$ and let $W=W_G.$ If $M'$ is another standard Levi subgroup then
any double coset in $W_M \bs W /W_{M'}$ has a unique element of
minimal length which we refer to as a left $W_M$ and right $W_{M'}$
reduced Weyl element. We denote by ${}_M W_{M'}$ the set of all left
$W_M$ and right $W_{M'}$ reduced Weyl elements. For integers $a$ and
$b$ we set $[a,b]=\{x \in \zz: a \le x \le b\}.$ For any subset $A
\subseteq [1,n]$ we denote by $SA$ the permutation group in the
elements of $A.$ It will be convenient to identify $W$ with
$S[1,n].$ If $P=MU$ and $P'=M'U'$ are standard parabolic subgroups
of $G$ with their standard Levi decompositions, the Bruhat
decomposition of $G$ gives the disjoint union
\begin{equation}\label{eq: bruhat decomp}
    G=\mathop{\sqcup}\limits_{w \in {}_M W_{M'}} Pw\overline {P'}.
\end{equation}

For any matrix $X$ let ${}^t X$ denote the transpose matrix.
%It will sometimes be
%convenient to denote by $T:G \to G$ the map $T(g)={}^t g.$
For a skew-symmetric matrix $\ii=-{}^t \ii \in G_{2k}$ let
\[
    Sp(\ii)=\{g \in G_{2k}:{}^t g\ii g=\ii\}
\]
and let
\[
 J_{2k}=
    \left(
\begin{array}{cc}
  0  & w_k   \\
  -w_k  & 0   \\
\end{array}
\right)
\]
where $w_k \in G_k$ is the matrix with $ij^{th}$ entry
$\delta_{i,n+1-j}.$ Denote by $U_r$ the subgroup of upper triangular
unipotent matrices and by $\overline U_r$ the subgroup of lower
triangular unipotent matrices in $G_r$. For non-negative integers
$r$ and $k$ let
\[
    H_{r,2k} =\{
\left(
\begin{array}{cc}
  u  & X   \\
  0  &  h  \\
\end{array}
\right): u \in U_r,\, X \in M_{r \times 2k}(F),\,h \in Sp(J_{2k})\}
\]
and let
\[
    \overline H_{r,2k} =\{
\left(
\begin{array}{cc}
  u  & 0   \\
  X  &  h  \\
\end{array}
\right): u \in \overline U_r,\, X \in M_{2k \times r}(F),\,h \in
Sp(J_{2k})\}.
\]
Note that $\overline H_{r,2k}$ is the image of $H_{r,2k}$ under
transpose. For $g \in G$ let
\[
g^\tau={}^t g^{-1}.
\]
The restriction to $ H_{r,2k}$ of the involution $\tau:G \to G$
defines a group isomorphism from $ H_{r,2k}$ to $\overline
H_{r,2k}.$ Let $n=r+2k=r'+2k'$ and let
$\hh^{r,r'}=\hh^{r,r'}_n=H_{r,2k} \times \overline H_{r',2k'}.$ Thus
\[
\hh^{r,r'}=\{(h_1,h_2^\tau):h_1 \in H_{r,2k},\,h_2 \in H_{r',2k'}\}.
\]
We denote by $e_{\hh^{r,r'}}$ the identity element of $\hh^{r,r'}.$
It will also be useful to consider the map $\xi:\hh^{r,r'} \to
\hh^{r',r}$ defined by
\[
    \xi(h_1,h_2^\tau)=(h_2,h_1^\tau).
\]
The group $\hh^{r,r'}$ acts on $G$ by
\[
    h \cdot g=h_1 g \, {}^t h_2,\, h=(h_1,h_2^\tau) \in \hh^{r,r'},\,g \in G.
\]
We observe that
\begin{equation}\label{eq: xi-transpose}
    {}^t(h \cdot g)=\xi(h) \cdot {}^t g,\,h \in \hh^{r,r'},\,g \in G.
\end{equation}
When $r=r'$ the map $\xi$ is an involution of $\hh^{r,r}.$ The
formula \eqref{eq: xi-transpose} allows us then to define the semi
direct product
\[
\widetilde \hh^{r,r}= \hh^{r,r} \rtimes \{\pm1\}
\]
with multiplication rule
\[
(h,\epsilon)(h',\epsilon')=(h
\,\xi_\epsilon(h'),\epsilon\epsilon')\text{ where
}\xi_\epsilon(h)=\begin{cases}h & \epsilon=1\\ \xi(h) & \epsilon=-1.
\end{cases}
\]
Here $h,\,h' \in \hh^{r,r},\,\epsilon,\,\epsilon' \in \{\pm 1\}.$
The group $\widetilde \hh^{r,r}$ acts on $G$ by
\[
    (h,\epsilon)\cdot g=h \cdot T_\epsilon (g) \text{ where }
    T_\epsilon(g)=\begin{cases}g & \epsilon=1\\ {}^tg & \epsilon=-1.
\end{cases}
\]
In order to unify notation, when $r \ne r' $ we shall set
$\widetilde \hh^{r,r'}=\hh^{r,r'} \times \{1\}.$

For a non-trivial character $\psi$ of $F$ we define as in
\S\ref{intro} the generic character $\psi_r$ of $U_r$ by \eqref{eq:
gen char} and the character $\psi_{r,2k}$ of $H_{r,2k}$ by
\eqref{eq: char on H}. Let $\theta^{r,r'}$ be the character of
$\hh^{r,r'}$ defined by
\[
    \theta^{r,r'}(h_1,h_2^\tau)=\psi_{r,2k}(h_1)\psi_{r',2k'}(h_2).
\]
We also extend $\theta^{r,r'}$ to the character $\widetilde
\theta^{r,r'}$ of $\widetilde \hh^{r,r'}$ defined by
\[
    \widetilde \theta^{r,r'}(h,\epsilon)=\epsilon\ \theta^{r,r'}(h).
\]

\section{Reduction to Invariant Distributions}\label{Reduction to Invariant Distributions}
Let $n=r+2k=r'+2k'$ be 2 decompositions of $n.$ Let $\hh=\hh^{r,r'}$
and $\theta=\theta^{r,r'}.$ The action of $\widetilde\hh$ on $G$
defines an action on $C_c^\infty(G)$ and on the space
$\dd(G)=C_c^\infty(G)^*$ of distributions on $G$ by
\[
    (h\cdot \phi)(g)=\phi(h^{-1}\cdot
    g)
\text{ and }
    (h\cdot D)(\phi)=D(h^{-1}\cdot
    \phi)
\]
for $h \in \widetilde \hh,\,g \in G,\,\phi \in C_c^\infty(G)$ and $D
\in \dd(G).$ In this section we show that Theorem
\ref{thm:uniq-disj} reduces to the following.
\begin{proposition}\label{prop: dist reduction}
If $D \in \dd(G)$ is such that $h \cdot D=\tilde\theta(h) D$ for all
$h \in \widetilde\hh$ then $D=0,$ i.e.
\begin{equation}\label{equivariant dist}
    \Hom_{\widetilde\hh}(C_c^\infty(G),\tilde\theta)=0.
\end{equation}
\end{proposition}
\subsection{Proposition \ref{prop: dist reduction} implies Theorem
\ref{thm:uniq-disj}}\label{subsec: prop to thm}
%Assume for now that Proposition \ref{prop: dist
%reduction} holds. We deduce Theorem \ref{thm:uniq-disj}.
Let $\pi$ be an irreducible representation of $G$. Set
$H=H_{r,2k},\,H'=H_{r',2k'},\,\psi=\psi_{r,2k}$ (forgive the abuse
of notation) and $\psi'=\psi_{r',2k'}.$ Denote by $\overline H$
(resp. $\overline{H'}$) the image of $H$ (resp. ${H'}$) under
$\tau.$ Let $\ell \in \Hom_{H}(\pi,\psi)$ and $\ell' \in
\Hom_{H'}(\pi,\psi').$ The representation $\pi^\tau(g)=\pi(g^\tau)$
realizes the contragradiant representation $\tilde \pi$ on the space
$V_\pi$ of $\pi$ \cite{MR0404534} (see also \cite[Theorem
7.3]{MR0425030}). Note that $\ell' \in
\Hom_{\overline{H'}}(\pi^\tau,(\psi')^\tau)$ defines a functional
$\tilde \ell'$ on the space $V_{\tilde \pi}$ of $\tilde\pi$ and that
$\tilde \ell' \in \Hom_{\overline{H'}}(\tilde\pi,(\psi')^\tau)$.
Note further that $\ell\circ\pi(\phi)$ is a smooth vector in
$V_{\tilde\pi}$. Define the distribution $D$ on $G$ by
\begin{equation}\label{dist}
    D(\phi)=\tilde\ell'(\ell\circ \pi(\phi)),\,\phi \in
    C_c^\infty(G).
\end{equation}
For $h \in H$ and $h' \in H'$ we have $\pi((h^{-1},{}^t h')\cdot
\phi)=\pi(h)\circ\pi(\phi)\circ\pi({}^th')$ and therefore
\[
    ((h,(h')^\tau)\cdot
    D)(\phi)=\tilde\ell'(\ell \circ\pi(h)\circ\pi(\phi)\circ\pi({}^t
    h')).
\]
By our assumption on $\ell$ and $\ell'$ we have, $\ell
\circ\pi(h)=\psi(h) \ell$ and $\tilde\ell' \circ \tilde
\pi((h')^\tau)=\psi'(h')\tilde \ell',\,h \in H,\,h' \in H'.$ Also
note that for any $\tilde v \in V_{\tilde \pi}$ viewed as a smooth
functional on $\pi$ the composition $\tilde v \circ \pi(g)$ is again
a smooth functional on $\pi$ and in fact
\[
    (\tilde v \circ \pi(g))(v)=\tilde
    v(\pi(g)v)=(\tilde\pi(g^{-1})\tilde v)(v)
\]
i.e.,
\[
    \tilde v \circ \pi(g)=\tilde\pi(g^{-1})\tilde v.
\]
Applying this to $\tilde v=\ell \circ \pi(\phi)$ and $g={}^t h'$ we
get that
\begin{multline*}
    ((h,(h')^\tau)\cdot
    D)(\phi)=\psi(h)\ \tilde\ell'((\ell \circ \pi(\phi))\circ \pi({}^t
    h'))\\=\psi(h)\ \tilde\ell'(\tilde\pi((h')^\tau)(\ell \circ \pi(\phi)))=
    \theta(h,(h')^\tau)\ D(\phi).
\end{multline*}
We see that $D$ is $(\hh,\theta)$-equivariant. If $r\ne r'$ it
follows from Proposition \ref{prop: dist reduction} that $D=0$. If
we assume further that $\ell$ is non-zero then the vectors $\ell
\circ \pi(\phi),\,\phi\in C_c^\infty(G)$ span $V_{\tilde \pi}$. We
conclude that $\tilde\ell'$ must vanish identically on $V_{\tilde
\pi}$ and hence also $\ell'=0.$ This shows that
\begin{equation}\label{eq: disjointness}
\dim_\cc(
\Hom_{H_{r,2k}}(\pi,\psi_{r,2k}))\dim_\cc(\Hom_{H_{r',2k'}}(\pi,\psi_{r',2k'}))=0
\text{ whenever }r\ne r'.
\end{equation}
Assume now that  $r=r'.$ Recall that $e_{\hh}$ is the unit element
of $\hh$. Note that $(e_\hh,-1)\cdot \phi={}^t \phi$ where
${}^t\phi(g)=\phi({}^tg),\,\phi \in C_c^\infty(G),\,g \in G.$ Note
further that for every $h \in \hh$ we have
\[
(h,1)(e_\hh,-1)=(e_\hh,-1)(\xi(h),1)
\]
and that $\theta(\xi(h))=\theta(h).$ Since $D \in
\Hom_\hh(C_c^\infty(G),\theta)$, it also follows that
\[
D_1=D-(e_\hh,-1)\cdot D \in\Hom_\hh(C_c^\infty(G),\theta).
\]
Furthermore, since $\tilde\theta(e_\hh,-1)=-1$ and $(e_\hh,-1)\cdot
D_1=-D_1$ we conclude that $D_1 \in
\Hom_{\widetilde\hh}(C_c^\infty(G),\tilde\theta).$ Proposition
\ref{prop: dist reduction} now implies that
\begin{equation}\label{eq: equivarience}
D=(e_\hh,-1)\cdot D.
\end{equation}
Let $B:C_c^\infty(G) \times C_c^\infty(G) \to \cc$ be the bilinear
form defined by
\begin{equation}
B(\phi_1,\phi_2)=D(\phi_{1}*\phi_{2})
\end{equation}
where
\[
    (\phi_{1}*\phi_{2})(g)=\int_G \phi_1(x) \phi_2(x^{-1}g)\ dx.
\]
Note that
\[
\pi(\phi_1 *\phi_2)=\pi(\phi_1)\circ \pi(\phi_2) \text{ and }
{}^t(\phi_1 *\phi_2)={}^t \phi_2 * {}^t\phi_1,\,\phi_1,\,\phi_2 \in
C_c^\infty(G).
\]
Thus, \eqref{eq: equivarience} implies that
\[
    B(\phi_1,\phi_2)=B((e_\hh,-1)\cdot \phi_2,(e_\hh,-1)\cdot\phi_1).
\]
This implies that $R_B=(e_\hh,-1)\cdot L_B$ where
\[
    L_B=\{\phi \in C_c^\infty(G): B(\phi,\cdot)\equiv 0\} \text{ and
    }
        R_B=\{\phi \in C_c^\infty(G): B(\cdot,\phi)\equiv 0\}
\]
are respectively the left and right kernels of $B.$ In other words
\begin{equation}\label{eq: left-right ker}
R_B=\{{}^t \phi:\phi \in L_B\}.
\end{equation}
For a functional $\lambda$ on $V_\pi$ let
\[
    \Ker(\lambda,\pi)=\{\phi \in C_c^\infty(G): \lambda \circ \pi(\phi)=0\}.
\]
Note that
\[
    B(\phi_1,\phi_2)=(\tilde \ell' \circ\pi(\phi_2^\vee))(\ell \circ
    \pi(\phi_1))
\]
where
\[
    \phi^\vee(g)=\phi(g^{-1})
\]
and therefore
\[
    L_B=\Ker(\ell,\pi) \text{ and } R_B=\{\phi^\vee:\phi \in \Ker(\tilde\ell',\tilde\pi)\}.
\]
By our definitions we have
\[
    \Ker(\tilde\ell',\tilde\pi)=\Ker(\ell',\pi^\tau)=\{({}^t\phi)^\vee:\phi \in\Ker(\ell',\pi)\}
\]
and therefore
\[
    R_B=\{{}^t\phi:\phi \in \Ker(\ell',\pi)\}.
\]
It now follows from \eqref{eq: left-right ker} that
\[
    \Ker(\ell,\pi)=\Ker(\ell',\pi).
\]
Since $\pi$ is irreducible we get that $\ker \ell=\ker\ell'$ and
therefore that $\ell$ and $\ell'$ are proportional. We therefore
proved that
\begin{equation}\label{eq: uniqueness}
    \dim_\cc( \Hom_{H_{r,2k}}(\pi,\psi_{r,2k})) \le 1 \text{ for all }0 \le k \le [\frac n2].
\end{equation}
Theorem \ref{thm:uniq-disj} is now a straightforward consequence of
\eqref{eq: frob}, \eqref{eq: disjointness} and \eqref{eq:
uniqueness}.

\section{Reduction to $\hh$-orbits}\label{sec: Reduction to
H-orbits} We keep the notation introduced in \S\ref{Reduction to
Invariant Distributions}. For every $g \in G$ we denote by $\hh_g$
the stabilizer of $g$ in $\hh$ and by $\widetilde \hh_g$ the
stabilizer of $g$ in $\widetilde \hh.$ The purpose of this section
is to reduce Proposition \ref{prop: dist reduction} to the
following.
\begin{proposition}\label{prop: main double cosets}
For every $g \in G$ the character $\widetilde \theta$ is non-trivial
on $\widetilde \hh_g.$
%\begin{equation}\label{eq: stab property}
%\text{the character }\widetilde \theta^{r,r'} \text{ is not trivial on }\widetilde \hh^{r,r'}_g.
%\end{equation}
\end{proposition}
\begin{remark}
The objects involved and the statement of Proposition \ref{prop:
main double cosets} makes sense over any field $F$ and in fact, our
proof is valid in this generality. In particular, using Mackey
theory, it can provide an alternative proof of the uniqueness and
disjointness of Klyachko models over a finite field.
\end{remark}
\subsection{Proposition \ref{prop: main double cosets} implies
Proposition \ref{prop: dist reduction}}\label{subsec: prop2 to
prop1} Assume now that Proposition \ref{prop: main double cosets}
holds. We deduce that Proposition \ref{prop: dist reduction} also
holds. Let $\ch_{\widetilde\hh_g}$ denote the trivial character of
$\widetilde\hh_g.$ Note that $h \cdot g \mapsto \widetilde\hh_g
\,h^{-1} $ is a homeomorphism of $\widetilde\hh$-spaces
$\widetilde\hh \cdot g \simeq \widetilde\hh_g\bs\widetilde\hh$ that
induces an $\widetilde\hh$-isomorphism
\begin{equation}\label{isom of orbit}
    C_c^\infty(\widetilde\hh \cdot g) \simeq \ind_{\widetilde\hh_g}^{\widetilde\hh}(\ch_{\widetilde\hh_g})
\end{equation}
where $\ind$ denotes smooth induction with compact support.
Therefore, by Frobenius reciprocity \cite[\S2.29]{MR0425030}
\begin{equation}\label{frobenious}
    \Hom_{\widetilde\hh}(C_c^\infty(\widetilde\hh \cdot g),\tilde\theta)=\Hom_{\widetilde\hh_g}
    (\delta_{\widetilde\hh_g},\theta_{|\widetilde\hh_g})
\end{equation}
where $\delta_{\widetilde\hh_g}$ is the modulus function of
$\widetilde\hh_g.$ Since the image of $\tilde\theta$ lies in the
unit circle (in fact, the image of $\theta$ lies in the group of
$p$-powered roots of unity where $p$ is the residual characteristic
of $F$) and since $\delta_{\widetilde\hh_g}$ is positive, we get
that whenever $\tilde\theta_{|\widetilde\hh_g}$ is non-trivial we
also have
\begin{equation}\label{eq: char inequal}
\tilde\theta_{|\widetilde\hh_g}\ne
 \delta_{\widetilde\hh_g}.
\end{equation}
It follows from Proposition \ref{prop: main double cosets} that
\eqref{eq: char inequal} holds for every $g \in G$ and therefore by
\eqref{isom of orbit} that
\begin{equation}\label{vanish for orbit}
\Hom_{\widetilde\hh}(C_c^\infty(\widetilde\hh \cdot
g),\tilde\theta)=0,\,g \in G.
\end{equation}
Proposition \ref{prop: dist reduction} follows from \eqref{vanish
for orbit} using the theory of Gelfand-Kazhdan \cite{MR0404534}.
Indeed, we apply \cite[Theorem 6.9]{MR0425030} to the following
setting. We view $C_c^\infty(G)$ as a module over itself by
convolution. By \cite[Proposition 1.14]{MR0425030} it uniquely
defines a sheaf $\mathcal{F}$ over the {$l$-}space $G.$ We let
$\widetilde\hh$ act on $C_c^\infty(G)$ by
\[
    h \cdot_{\tilde\theta} \phi=\tilde\theta(h) h \cdot \phi.
\]
This defines an action of $\widetilde\hh$ on the sheaf
$\mathcal{F}.$ The space of $\widetilde\hh$-invariant distributions
on $\mathcal{F}$ is then precisely
$\Hom_{\widetilde\hh}(C_c^\infty(G),\tilde\theta).$ The action of
$\widetilde\hh$ on $G$ is constructible by \cite[\S 6.15, Theorem
A]{MR0425030}. The second assumption of \cite[Theorem
6.9]{MR0425030} is precisely \eqref{vanish for orbit}. It follows
that there are no $\widetilde\hh$-invariant distributions on the
sheaf $\mathcal{F},$ i.e. that \eqref{equivariant dist} holds.
\section{The property of $\hh$-orbits made explicit}\label{sec:
explicit property}

In order to prove Proposition \ref{prop: main double cosets} it will
be convenient to reformulate it, by describing more explicitly the
property of the $\widetilde\hh$-orbits that we wish to prove. We
begin with this reformulation.

\subsection{The property $\prop(g,r,r')$}

For $g \in G$ let $\prop(g,r,r')=\prop_n(g,r,r')$ be the following
property: either
\begin{equation}\label{eq: non triv char}
\text{ there exists } y \in H_{r,2k} \text{ such that } g^{-1}yg \in
\overline H_{r',2k'} \text{ and }\theta^{r,r'}(y,g^{-1}yg ) \ne 1
\end{equation}
or $r=r'$ and
\begin{equation}\label{eq: transpose inv orbit}
\text{ there exists } y \in H_{r,2k} \text{ such that } g^{-1}y\
{}^tg \in \overline H_{r,2k} \text{ and }\theta^{r,r}(y,g^{-1} y \
{}^tg ) = 1.
\end{equation}
\begin{lemma}\label{lemma: reformulate prop}
For every $g \in G,$ $\widetilde\theta^{r,r'}$ is non-trivial on
$\widetilde\hh^{r,r'}_g$ if and only if $\prop(g,r,r').$
\end{lemma}
\begin{proof}
Note that
\[
    \hh^{r,r'}_g=\{(y,g^{-1}yg):y \in H_{r,2k} \cap g\overline
    H_{r',2k'}g^{-1}\}
\]
and therefore \eqref{eq: non triv char} holds if and only if
$\theta^{r,r'}$ is not trivial on $\hh^{r,r'}_g.$ If $r \ne r'$ this
proves the lemma. If $r=r'$ it remains to show that when
$\theta^{r,r}$ is trivial on $\hh^{r,r}_g$ then
$\widetilde\theta^{r,r}$ is not trivial on $\widetilde\hh^{r,r}$ if
and only if we have \eqref{eq: transpose inv orbit}. Note that
\[
    \{h \in \hh^{r,r}:(h,-1) \in \widetilde
    \hh^{r,r}_g\}=\{(y,g^{-1} y\ {}^tg):y \in H_{r,2k} \cap
     g \overline
    H_{r',2k'}g^{\tau}\}.
\]
If $y \in H_{r,2k} \cap
     g \overline
H_{r',2k'}g^{\tau}$ then for $h=(y,g^{-1} y\ {}^tg) \in \hh^{r,r}$
we have $h \cdot {}^t g=g$ and therefore by \eqref{eq: xi-transpose}
we get that $h\,\xi(h) \in \hh^{r,r}_g$ so that
$\theta^{r,r}(h\,\xi(h))=1.$ Since $\theta^{r,r}=\theta^{r,r}\circ
\xi$ we have $\theta^{r,r}(h)\in \{\pm1\}.$ The remaining of the
lemma follows.
\end{proof}
We make here another simple observation that will help to shorten
some of the arguments in the proof of Proposition \ref{prop: main
double cosets}.
\begin{lemma}\label{lemma: reduce to representative+transpose}
If $\prop(g,r,r')$ then $\prop(h\cdot g,r,r')$ for all $h \in
\widetilde\hh$  and $\prop({}^t g,r',r).$
\end{lemma}
\begin{proof}
Note that $\widetilde\hh_{h \cdot g}=h \widetilde\hh_g h^{-1}$ and
that $\widetilde \theta$ is a character. Thus, the first statement
is immediate from Lemma \ref{lemma: reformulate prop}. If $r=r'$
this argument with $h=(e_\hh,-1)$ also contains the second
statement. If $r\ne r'$ the second statement follows from the fact
that $\hh^{r',r}_{^tg}=\xi(\hh^{r,r'}_g)$ (that follows from
\eqref{eq: xi-transpose} ) and the fact that $\theta \circ
\xi=\theta .$
\end{proof}

In light of Lemma \ref{lemma: reformulate prop} in order to show
Proposition \ref{prop: main double cosets} we need to show that for
every $r,\,r' \le n$ such that $n-r \equiv n-r' \equiv 0(\mod 2)$
and for every $g \in G$ we have $\prop(g,r,r').$ This will occupy
the rest of this paper.
\subsection{Two cases where $\prop(g,r,r')$ is already known}
There are two extremes that are already known. The first is a well
known fact concerned with the double coset space $U_n \bs
G/\overline U_n.$ It can be found in the proof of \cite[Lemma
4.3.8]{MR0404534} (it is essentially the steps (a)-(d) verifying
condition 4 of \cite[Theorem 4.2.10]{MR0404534}) and it is applied
in order to prove the uniqueness of Whittaker models.
\begin{lemma}\label{lemma: whittaker case}
For every $g \in G$ we have $\prop_n(g,n,n).$
\end{lemma}
The second extreme is with respect to the symplectic group. It was
proved by Heumus and Rallis \cite[Proposition 2.3.1]{MR1078382}
based on results of Klyachko \cite[Corollary 5.6]{kly}. Recently,
Goldstein and Guralnick essentially provided an independent proof
over any field \cite[Proposition 3.1]{MR2290925}.
\begin{lemma}\label{lemma: symplectic case}
When $n$ is even for every $g \in G$ we have $\prop_n(g,0,0).$
\end{lemma}
\begin{proof}
We show that when $r=r'=0$ \eqref{eq: transpose inv orbit} holds for
every $g \in G.$ That is, we show that for every $g \in G$ we have
${}^t g \in Sp(J_n)gSp(J_n).$ As observed in the proof of Lemma
\ref{lemma: reduce to representative+transpose}, it is enough to
prove that there exists $y \in Sp(J_n)gSp(J_n)$ such that ${}^t y
\in Sp(J_n)gSp(J_n).$ Let $n=2k$ and let
\[
    J'_n=\left(
           \begin{array}{cc}
             0 & I_k \\
             -I_k & 0 \\
           \end{array}
         \right)={}^t \sigma J_n \sigma \text{ where }\sigma=\left(
           \begin{array}{cc}
             w_k & 0 \\
             0 & I_k \\
           \end{array}
         \right).
\]
Thus,
\[
    Sp(J_n')=\sigma^{-1}Sp(J_n)\sigma.
\]
It follows from \cite[Proposition 3.1]{MR2290925} that there exists
$g' \in G_k$ such that $\diag(I_k,g') \in Sp(J_n')\sigma^{-1}
g\sigma Sp(J_n'),$ i.e that $y=\sigma\diag(I_k,g')\sigma^{-1} \in
Sp(J_n) g Sp(J_n).$ Since every matrix in $G_k$ is conjugate to its
transpose and since $\diag(x,{}^t x) \in Sp(J_n')$ for every $x \in
G_k$ we see that $\diag(I_k,{}^t g') \in Sp(J_n')\sigma g\sigma
Sp(J_n'),$ i.e. that ${}^t y=\sigma\diag(I_k,{}^t g')\sigma^{-1} \in
Sp(J_n) g Sp(J_n).$
\end{proof}

\section{Proof by induction of $\prop_n(g,r,r')$}\label{sec: induction} Fix 2 decompositions $n=r+2k=r'+2k'.$
We prove by induction on $n$ that for every $g \in G$ we have
$\prop_n(g,r,r').$ If $r=r'=0$ then this is Lemma \ref{lemma:
symplectic case}. We assume from now on that $r+r'>0.$ By the
induction hypothesis we may also assume that for all $n_1<n,$ all
$r_1,\,r_1' \le n_1$ such that $n_1-r_1 \equiv n_1-r_1' \equiv 0
(\mod 2)$ and all $g' \in G_{n_1}$ we have
$\prop_{n_1}(g',r_1,r_1').$ Set $H=H_{r,2k},\,H'=H_{r',2k'},\,\hh=H
\times \overline H'$ and $\theta=\theta^{r,r'}.$ Let
$P=P_{(1^{(r)},2k)}$ and $P'=P_{(1^{(r')},2k')}.$ For $w \in W$
viewed as a permutation in $S[1,n]$ let
\[
    I_w=\{i \in [1,r]: w^{-1}(i) \in [1,r']\}.
\]
\subsection{A simple proof for most Bruhat cells}
\begin{lemma}\label{lemma: induction step}
Let $w \in {}_M W_{M'}$ be such that $I_w$ is not empty then
$\prop(g,r,r')$ holds for every $g \in Pw\overline{P'}.$
\end{lemma}
\begin{proof}
Note that $U \times \overline U'\subseteq \hh$ and therefore that
every $\hh$-orbit in $Pw\overline{P'}$ contains an element of
$MwM'.$ In light of Lemma \ref{lemma: reduce to
representative+transpose} we may assume without loss of generality
that $g \in MwM'.$

Assume first that there exists an integer $i$ such that $1 \le i \le
\min\{r,r'\}$ and $I_w=w^{-1}(I_w)=[1,i].$ We can then write
$w=\diag(w_1,w_2)$ for some $w_1 \in S[1,i]$ and $w_2 \in S[i+1,n].$
Thus for $g \in MwM'$ there exists $g_1,\,g_2 \in G_{n-i},\,$ and
$a=\diag(a_1,\dots,a_i) \in G_i$ such that
$g=\diag(I_i,g_1)w\diag(a,g_2)=\diag(w_1a,g')$ for $g'=g_1w_2g_2 \in
G_{n-i}.$ Let $(u_1,u_2^\tau,\epsilon) \in (\widetilde
\hh^{i,i}_i)_{w_1a} $ be such that
$\widetilde\theta^{i,i}(u_1,u_2^\tau,\epsilon) \not=1$ and let
$(h_1,h_2^\tau,\epsilon') \in (\widetilde
\hh^{r-i,r'-i}_{n-i})_{g'}$ be such that
$\widetilde\theta^{r-i,r'-i}(h_1,h_2^\tau,\epsilon') \ne 1.$ The
first exists by Lemma \ref{lemma: whittaker case}. For the second we
apply the induction hypothesis to have $\prop_{n-i}(g',r-i,r'-i).$
If $\epsilon=1$ then
\[
    h=(\diag(u_1,I_{n-i}),\diag(u_2,I_{n-i})^\tau,1) \in \widetilde\hh_g \text{ and }
\widetilde\theta(h)=\widetilde\theta^{i,i}(u_1,u_2^\tau,1) \not=1.
\]
Similarly, if $\epsilon'=1$ then
\[
h=(\diag(I_i,h_1),\diag(I_i,h_2)^\tau,1) \in \widetilde\hh_g\text{
and }
\widetilde\theta(h)=\widetilde\theta^{r-i,r'-i}(h_1,h_2^\tau,1)
\not=1.
\]
If on  the other hand $\epsilon=\epsilon'=-1$ then
\[
h=(\diag(u_1,h_1),\diag(u_2,h_2)^\tau,-1) \in \widetilde\hh_g\text{
and } \widetilde\theta(h)=-1.
\]
We are now left with the case that either $I_w$ or $w^{-1}(I_w)$ is
not of the form $[1,i]$ as above. Note that if $g \in
Pw\overline{P'}$ then ${}^t g \in P' w^{-1} \overline P$ and that
$w^{-1} \in {}_{M'}W_M.$ It follows from Lemma \ref{lemma: reduce to
representative+transpose} that it is enough to prove our lemma
either for $g$ or for ${}^t g.$ We may therefore assume, without
loss of generality, that $I_w$ is not of the form $[1,i]$ for any $1
\le i \le \min\{r,r'\}.$ Since we assume that $g \in MwM'$ there
exist $g_1 \in G_{2k},\,g_2 \in G_{2k'}$ and $a
=\diag(a_1,\dots,a_{r'})$ a diagonal matrix in $G_{r'}$ such that
$g=\diag(I_r,g_1)w\diag(a,g_2).$ By our assumption on $w$ we have
that $[1,r] \setminus I_w$ is not empty. Let $\ell=\min([1,r]
\setminus I_w).$ Since $[1,\ell-1]$ is contained but does not equal
$I_w$ the set $[\ell+1,r] \cap I_w$ is not empty. Let
$q=\min([\ell+1,r] \cap I_w).$ Then $q-1 \not\in I_w$ and $q \in
I_w$. In particular, $w^{-1}(q-1)>r'$ and $w^{-1}(q) \le r'.$ Let
$E_{i,j} \in M_{n \times n}(F)$ be the matrix with $(b,c)^{th}$
entry equal to $\delta_{(i,j),(b,c)}$ and let
$u_{i,j}(s)=I_n+s\,E_{i,j},\, s \in F.$ Note that $u_{q-1,q}(s) \in
U \subseteq H_{r,2k}$ and that $\psi_{r,2k}(u_{q-1,q}(s))=\psi(s)$.
Thus, there exists $s \in F$ such that $\psi_{r,2k}(u_{q-1,q}(s))
\ne 1.$ On the other hand,
\begin{multline*}
g^{-1} u_{q-1,q}(s)g=\left(
\begin{array}{cc}
a^{-1} & 0 \\
0 & g_2^{-1} \\
\end{array}
\right)u_{w^{-1}(q-1),w^{-1}(q)}(s) \left(
\begin{array}{cc}
a & 0 \\
0 & g_2 \\
\end{array}
\right)\\= \left(
\begin{array}{cc}
I_{r'} & 0 \\
* & I_{2k'} \\
\end{array}
\right) \in \overline H_{r',2k'}
\end{multline*}
and $\psi_{r',2k'}(g^{-1} u_{q-1,q}(s)g)=1.$ It follows that
$h_s=(u_{q-1,q}(s),g^{-1} u_{q-1,q}(s)g)\in \hh_g$ and if $s$ is
such that $\psi_{r,2k}(u_{q-1,q}(s)) \ne 1$ then $\theta(h_s)\ne 1.$
\end{proof}
\subsection{The closed Bruhat cell}\label{sec: closed cell}
We are now left with the case that $I_w$ is empty. Since this means
that $w^{-1}$ maps $[1,r]$ into $[r'+1,n]$ we must have, in
particular, $n \ge r+r'.$ It is not difficult to see that there is
then a unique such element in ${}_MW_{M'}$, namely,
\[
    w=w^{r,r'}=\left(
        \begin{array}{ccc}
          0 & I_r & 0 \\
          I_{r'} & 0 & 0 \\
          0 & 0 & I_{n-(r+r')} \\
        \end{array}
      \right).
\]
Note then that $ Pw \overline{P'},$ is the closed Bruhat cell. We
remark further that this contains the case that either $r$ or $r'$
is $0.$ Let $g \in MwM'.$ Note that there exist  $g_1 \in G_{2k}$
and $g_2 \in G_{2k'}$ such that
\[
g=\left(
      \begin{array}{cc}
        I_r &  \\
         & g_1 \\
      \end{array}
    \right)w\left(
      \begin{array}{cc}
        I_{r'} &  \\
         & g_2 \\
      \end{array}
    \right).
\]
Indeed, for $t \in G_r,\,t' \in G_{r'}$ (and in particular when $t$
and $t'$ are diagonal) if $g_1' \in G_{2k}$ and $g_2' \in G_{2k'}$
we have
\begin{multline*}
    \left(
      \begin{array}{cc}
        t &  \\
         & g_1' \\
      \end{array}
    \right)w
    \left(
      \begin{array}{cc}
        t' &  \\
         & g_2' \\
      \end{array}
    \right)\\=\left(
      \begin{array}{cc}
        I_r &  \\
         & g_1' \\
      \end{array}
    \right)\left(\begin{array}{ccc}
          0 & t & 0 \\
          t' & 0 & 0 \\
          0 & 0 & I_{n-(r+r')} \\
        \end{array}
      \right)\left(
      \begin{array}{cc}
        I_r' &  \\
         & g_2' \\
      \end{array}
    \right)=\left(
      \begin{array}{cc}
        I_r &  \\
         & g_1 \\
      \end{array}
    \right)w
    \left(
      \begin{array}{cc}
        I_{r'} &  \\
         & g_2 \\
      \end{array}
    \right)
\end{multline*}
where $g_1=g_1'\diag(t',I_{2k-r'})$ and
$g_2=\diag(t,I_{2k'-r})g_2'.$

In order to show $\prop(g,r,r')$ we distinguish between 2 cases. We
denote by $\sprod{v_1}{\dots,v_i}$ the subspace of a vector space
$V$ spanned by  $v_1,\dots,v_i \in V.$ Let $V$ be a subspace of the
vector space $M_{\ell \times 1}(F)$ for some positive integer
$\ell.$ We say that a skew symmetric matrix $\ii \in M_{\ell \times
\ell}(F)$ is totally isotropic on $V$ if ${}^t v\ii v'=0$ for all
$v,\,v' \in V.$ Denote by $e_i$ the column vector with $1$ in the
$i^{th}$ row and $0$ in each other row. Thus $e_i \in M_{\ell \times
1}(F)$ for an integer $\ell$ which is implicit in our notation. Let
\[
\ii_1={}^t g_1J_{2k}g_1\text{ and }\ii_2=g_2^\tau J_{2k'}g_2^{-1}.
\]
We say that $g$ belongs to the totally isotropic case if both
$\ii_1^{-1}$ is totally isotropic on $\sprod{e_1}{\dots,e_{r'}}$ and
$\ii_2$ is totally isotropic on $\sprod{e_1}{\dots,e_{r}}.$
Otherwise we say that $g$ does not belong to the totally isotropic
case. It is easy to verify that this property indeed depends only on
$g$ and not on $g_1$ and $g_2.$ We now prove $\prop(g,r,r')$
separately in each of the 2 cases.

\subsubsection{When $g$ does not belong to the totally isotropic case}
In this case we prove that $g$ satisfies \eqref{eq: non triv char}.
It will be convenient to make this property more explicit. We say
that the 2 skew-symmetric forms $\ii_1,\,\ii_2 \in G$ satisfy the
property $\qprop(\ii_1,\ii_2,r,r')$ if there exist $u \in U_r$ and
$u' \in U_{r'}$ such that $\psi_r(u)\ne \psi_{r'}(u')$ and for some
$X \in M_{r \times 2k'-r}(F),\,Y \in M_{r'
      \times 2k-r'}(F)$ and $D \in G_{n-(r+r')}$ we have
\[
      \left(
           \begin{array}{cc}
           u & X \\
           0 & D
        \end{array}
      \right) \in  Sp(\ii_2)
\text{ and }
    \left(
           \begin{array}{cc}
           {}^t u' & 0 \\
           Y & D
        \end{array}
      \right) \in  Sp(\ii_{1}).
\]
\begin{lemma}\label{lemma: q equiv}
Let
\[
g=\left(
      \begin{array}{cc}
        I_r &  \\
         & g_1 \\
      \end{array}
    \right)w\left(
      \begin{array}{cc}
        I_{r'} &  \\
         & g_2 \\
      \end{array}
    \right) \in MwM'
\]
%not belong to the totally isotropic case
and let
\[
\ii_1={}^t g_1J_{2k}g_1\text{ and }\ii_2=g_2^\tau J_{2k'}g_2^{-1}.
\]
Then $g$ satisfies \eqref{eq: non triv char} if and only if
$\qprop(\ii_1,\ii_2,r,r').$
\end{lemma}
\begin{proof}
Let
\[
    y=\left(
      \begin{array}{cc}
        u & Z \\
         & h \\
      \end{array}
    \right)\in H
\]
with $u \in U_r,\,h \in Sp(J_{2k})$ and $Z \in M_{r \times 2k}(F).$
To explicate condition \eqref{eq: non triv char} we compute
$g^{-1}yg.$ First note that we have
\[
    \left(
      \begin{array}{cc}
        I_r &  \\
         & g_1^{-1} \\
      \end{array}
    \right)\left(
      \begin{array}{cc}
        u & Z \\
         & h \\
      \end{array}
    \right)\left(
      \begin{array}{cc}
        I_r &  \\
         & g_1 \\
      \end{array}
    \right)=\left(
      \begin{array}{cc}
        u & Zg_1 \\
         & g_1^{-1}hg_1 \\
      \end{array}
    \right).
\]
We write
\[
    g_1^{-1}hg_1=\left(
                   \begin{array}{cc}
                     {}^tu' & B \\
                     Y & D \\
                   \end{array}
                 \right) \text{ and }
                 Zg_1=(Z_1,Z_2)
\]
with $u' \in M_{r' \times r'}(F),\,D \in M_{2k-r' \times
2k-r'}(F),\,Z_1 \in M_{r \times r'}(F)$ and $Z_2 \in M_{r \times
2k-r'}(F).$ We then have
\[
    \left(
        \begin{array}{ccc}
          0 & I_{r'} & 0 \\
          I_{r} & 0 & 0 \\
          0 & 0 & I_{n-(r+r')} \\
        \end{array}
      \right)\left(
        \begin{array}{ccc}
          u & Z_1 & Z_2 \\
          0 & {}^tu' & B \\
          0 & Y & D \\
        \end{array}
      \right)\left(
        \begin{array}{ccc}
          0 & I_r & 0 \\
          I_{r'} & 0 & 0 \\
          0 & 0 & I_{n-(r+r')} \\
        \end{array}
      \right)=\left(
        \begin{array}{ccc}
          {}^tu' & 0 & B \\
          Z_1 & u & Z_2 \\
          Y & 0 & D \\
        \end{array}
      \right).
\]
Therefore,
\[
    g^{-1}yg=\left(
        \begin{array}{cc}
          {}^tu' & (0 , B)g_2 \\
          g_2^{-1}\left(
            \begin{array}{c}
          Z_1 \\Y \end{array}\right)
           & g_2^{-1}\left(
           \begin{array}{cc}
           u & Z_2 \\
           0 & D
        \end{array}
      \right)g_2\end{array}\right) .
\]
We see that $g^{-1}yg\in \overline{H'}$ if and only if $u' \in
 U_{r'},\,B=0$ and
\[
g_2^{-1}\left(
           \begin{array}{cc}
           u & Z_2 \\
           0 & D
        \end{array}
      \right)g_2 \in Sp(J_{2k'}).
\]
Recall also that
\[
    \left(
           \begin{array}{cc}
           {}^tu' & 0 \\
           Y & D
        \end{array}
      \right) \in g_1^{-1} Sp(J_{2k})g_1.
\]
With this notation, when $g^{-1}yg\in \overline{H'}$ we have
\[
    \theta(y,g^{-1}yg)=\psi_r(u)\psi_{r'}((u')^{-1}).
\]
Since
\[
    g_1^{-1} Sp(J_{2k})g_1=Sp(\ii_1) \text{ and }g_2
    Sp(J_{2k})g_2^{-1}=Sp(\ii_2),
\]
the lemma is now immediate.
\end{proof}
In order to proceed we need the following Lemma of Klyachko
\cite[\S1.3, p. 368, Step 3]{kly}.

\begin{lemma}\label{lemma: unipotent simplectic}
Let $\ii=-{}^t \ii\in G_{2k}$ and let $r \le 2k$ be such that $\ii$
is not totally isotropic on $\sprod{e_1}{\dots,e_r}$ then there
exists $u \in U_r$ with $\psi_r(u) \ne 1$ and $X \in M_{r \times
2k-r}(F)$ such that
\begin{equation}\label{eq: unip form}
    \left(
           \begin{array}{cc}
           u & X \\
           0 & I_{2k-r}
        \end{array}
      \right) \in Sp(\ii).
\end{equation}
\end{lemma}
\begin{proof}
Let $i \in [1,r-1]$ be maximal so that $\ii$ is totally isotropic on
$\sprod{e_1}{\dots,e_i}.$ There is therefore $v_0 \in
\sprod{e_1}{\dots,e_i}$ such that ${}^t v_0\ii e_{i+1} \ne0.$ We may
further assume that $v_0 \in e_i+\sprod{e_1}{\dots,e_{i-1}}$ since
if ${}^t e_i\ii e_{i+1} \ne0$ then we may take $v_0=e_i$ and
otherwise, we may replace $v_0$ by its sum with any scalar multiple
of $e_i.$ Let $V=M_{2k \times 1}(F)$ and for every $s \in F$ define
$\lambda_s \in \Hom_F(V,F)$ by $\lambda_s(v)=s\,{}^t v_0\ii v.$ Note
that the map $s \mapsto \lambda_s(e_{i+1}),\,s \in F$ is onto $F.$
Identify $GL(V)$ with $G_{2k}$ via the standard basis
$\{e_1,\dots,e_{2k}\}$ and define an element $h_s \in G_{2k}$ by
\[
    h_s(v)=v+ \lambda_s(v)\, v_0.
\]
Thus, $h_s \in Sp(\ii)$ is of the form \eqref{eq: unip form} with
$\psi_r(u)=\psi(\lambda_s(e_{i+1})).$
\end{proof}
\begin{lemma}\label{lemma: not totally isotropic}
Let
\[
g=\left(
      \begin{array}{cc}
        I_r &  \\
         & g_1 \\
      \end{array}
    \right)w\left(
      \begin{array}{cc}
        I_{r'} &  \\
         & g_2 \\
      \end{array}
    \right) \in MwM'
\]
not belong to the totally isotropic case and let
\[
\ii_1={}^t g_1J_{2k}g_1\text{ and }\ii_2=g_2^\tau J_{2k'}g_2^{-1}.
\]
Then we have $\qprop(\ii_1,\ii_2,r,r').$
\end{lemma}
\begin{proof}
If $\ii_2$ is not totally isotropic on $\sprod{e_1}{\dots,e_{r}}$
then by Lemma \ref{lemma: unipotent simplectic} there exist $u \in
U_r$ and $X \in M_{r\times 2k'-r}$ such that $\psi_r(u)\ne 1$ and
\[
\left(
           \begin{array}{cc}
           u & X \\
           0 & I_{2k'-r}
        \end{array}
      \right) \in Sp(\ii_2).
\]
Then $\qprop(\ii_1,\ii_2,r,r')$ is satisfies with $Y=0,\,u'=I_{r'}$
and $D=I_{n-(r+r')}.$ Note further that $Sp(\ii_1^{-1})=\{{}^t g:g
\in Sp(\ii_1)\} .$ Thus, if $\ii_1^{-1}$ is not totally isotropic on
$\sprod{e_1}{\dots,e_{r'}}$ then by Lemma \ref{lemma: unipotent
simplectic} applied to $\ii_1^{-1}$ there exist $u' \in U_{r'}$ and
$Y \in M_{2k-r' \times r'}$ such that $\psi_{r'}(u')\ne 1$ and
\[
\left(
           \begin{array}{cc}
           {}^t u' & 0 \\
           Y & I_{2k-r'}
        \end{array}
      \right) \in Sp(\ii_1).
\]
Thus, $\qprop(\ii_1,\ii_2,r,r')$ is satisfies with $X=0,\,u=I_r$ and
$D=I_{n-(r+r')}.$
\end{proof}

\subsubsection{When $g$ belongs to the totally isotropic case} Assume
from now on that both $\ii_2$ is totally isotropic on
$\sprod{e_1}{\dots,e_{r}}$ and $\ii_1^{-1}$ is totally isotropic on
$\sprod{e_1}{\dots,e_{r'}}.$ In the case at hand $\hh \cdot g$
contains an element of a rather simple form that will allow us the
inductive argument. In order to bring $g$ to this simpler form we
need the following lemma.
\begin{lemma}\label{lemma: max parabolic orbit}
Let $\ell \le m$ and $Q=P_{(\ell,2m-\ell)}.$ Then
\[
    Sp(J_{2m})\,Q=\{g \in G_{2m}:{}^tg J_{2m} g\text{ is totally
    isotropic on }\sprod{e_1}{\dots,e_{\ell}}\}.
\]
\end{lemma}
\begin{proof}
If $h \in Sp(J_{2m})$ and $q\in Q$ then $ {}^t(hq)J_{2m}hq={}^t
qJ_{2m}q.$ Since $q$ preserves the space
$\sprod{e_1}{\dots,e_{\ell}}$ and since $J_{2m}$ is totally
isotropic on $\sprod{e_1}{\dots,e_{\ell}}$ we get that ${}^t
qJ_{2m}q$ is also totally isotropic on
$\sprod{e_1}{\dots,e_{\ell}}.$ To prove the other direction let $g
\in G_{2m}$ be such that ${}^tg J_{2m} g$ is totally
    isotropic on $\sprod{e_1}{\dots,e_{\ell}}.$ Then
\[
    x=g^{t}J_{2m}g=\left(
           \begin{array}{cc}
           0_\ell & A \\
           -{}^t A & D
        \end{array}
      \right) \in G_{2m}
\]
for some $D=-{}^t D \in M_{2m-\ell \times 2m-\ell}(F).$ We must show
that there exists $q \in Q$ such that ${}^t qxq=J_{2m}.$ Since $x$
is invertible and $\ell \le 2m-\ell$ the matrix $A$ is of rank
$\ell.$ Performing elementary operations, there exists $\alpha \in
G_\ell$ and $\gamma \in G_{2m-\ell}$ such that ${}^t\alpha
A\gamma=(0_{\ell \times 2(m-\ell)},w_\ell).$ It follows that for
$q=\diag(\alpha,\gamma) \in Q,$ ${}^tqxq$ has the form
\[
\left(
  \begin{array}{ccc}
    0 & 0 & w_\ell \\
    0 & a & b \\
    -w_\ell & -{}^t b & d \\
  \end{array}
\right)
\]
where $a=-{}^ta \in G_{2(m-\ell)}$ and $d=-{}^t d \in M_{\ell \times
\ell}(F).$ Write $\beta=(\beta_1,\beta_2)$ with $\beta_1\in M_{\ell
\times 2(m-\ell)}(F)$ and $\beta_2 \in M_{\ell \times \ell}(F).$
Note that
\begin{multline*}
    \left(
           \begin{array}{ccc}
           I_{\ell} & 0 & 0\\
           {}^t \beta_1 & I_{2(m-\ell)} & 0 \\
           {}^t \beta_2 & 0 & I_\ell
        \end{array}
      \right)\left(
  \begin{array}{ccc}
    0 & 0 & w_\ell \\
    0 & a & b \\
    -w_\ell & -{}^t b & d \\
  \end{array}
\right)\left(
           \begin{array}{ccc}
           I_{\ell} & \beta_1 & \beta_2 \\
            0 & I_{2(m-\ell)} & 0 \\
            0 & 0 & I_\ell
        \end{array}
      \right)\\=\left(
  \begin{array}{ccc}
    0 & 0 & w_\ell \\
    0 & a & b+{}^t\beta_1w_\ell \\
    -w_\ell & -{}^t b-w_\ell\beta_1 & d+{}^t\beta_2w_\ell-w_\ell\beta_2 \\
  \end{array}
\right).
\end{multline*}
We may now take $\beta_1=-w_\ell {}^t b.$ Any skew symmetric matrix
in $M_{\ell \times \ell}(F)$ can be written as a difference
$X-{}^tX$ for some $X \in M_{\ell \times \ell}(F).$ Thus, there also
exists $\beta_2$ such that ${}^t\beta_2w_\ell-w_\ell\beta_2=-d.$ We
get that there exists $q \in Q$ such that
\[
    {}^t q x q=\left(
  \begin{array}{ccc}
    0 & 0 & w_\ell \\
    0 & a & 0 \\
    -w_\ell & 0 & 0 \\
  \end{array}
\right).
\]
Let $y \in G_{2(m-\ell)}$ be such that ${}^ty a y=J_{2(m-\ell)}.$
Thus $q'=q\diag(I_\ell,y,I_\ell) \in Q$ and ${}^t q'xq'=J_{2m}.$
\end{proof}
For $x \in G_\ell$ let
\[
\tilde x=w_\ell  \,x^\tau \,w_\ell.
\]
The following property of the group $Sp(J_{2m})$ will be used
several times in the proof of $\prop(g,r,r').$ Assume that $\ell \le
m.$
\begin{multline}\label{eq: symp property}
\text{For all } \,x \in G_\ell,\,s \in Sp(J_{2(m-\ell)})\text{ and
}y \text{ there exists } \,y^*  \text{ uniquely determined}\\
\text{by }x,\,s \text{ and }y \text{ and dependent linearly on
}y\text{ and there exists } z \text{ such that}\\
    \left(
      \begin{array}{ccc}
        x & y^* & z \\
        0 & s & y \\
        0 & 0 & \tilde x \\
      \end{array}
    \right) \text{ (resp.} \left(
      \begin{array}{ccc}
        x & 0 & 0 \\
        y^* & s & 0 \\
        z & y & \tilde x \\
      \end{array}
    \right))
     \text{ lies in } Sp(J_{2m}).
\end{multline}
We now choose a convenient representative for $g.$
\begin{lemma}\label{lemma: good representative}
Let
\[
g=\left(
      \begin{array}{cc}
        I_r &  \\
         & g_1 \\
      \end{array}
    \right)w\left(
      \begin{array}{cc}
        I_{r'} &  \\
         & g_2 \\
      \end{array}
    \right) \in MwM'
\]
belong to the totally isotropic case. Then there exists $\gamma \in
G_{n-(r+r')}$ such that
\[
    \left(
      \begin{array}{ccc}
        0 & I_r & 0 \\
        I_{r'} & 0 & 0 \\
        0 & 0 & \gamma \\
      \end{array}
    \right) \in \hh \cdot g.
\]
\end{lemma}
\begin{proof}
Since $-\ii_1^{-1}=g_1^{-1}J_{2k} g_1^\tau$ is totally isotropic on
$\sprod{e_1}{\dots,e_{r'}}$ and $\ii_2=g_2^{\tau}J_{2k} g_2^{-1}$ is
totally isotropic on $\sprod{e_1}{\dots,e_{r}}$, it follows from
Lemma \ref{lemma: max parabolic orbit} that
\[
g_1 \in Sp(J_{2k})\left(
           \begin{array}{cc}
           \alpha_1 & 0 \\
           \beta_1' & \gamma_1
        \end{array}
      \right)\text{ and }g_2 \in  \left(
           \begin{array}{cc}
           \alpha_2 & \beta_2' \\
           0 & \gamma_2
        \end{array}
      \right)Sp(J_{2k'})
\]
for some $\alpha_1 \in G_{r'},\,\gamma_1 \in G_{2k-r'},\,\alpha_2
\in G_{r},\,\gamma_2 \in G_{2k'-r}$ and $\beta_1'$ and $\beta_2'$ of
the appropriate size. Therefore,
\begin{multline*}
    \left(
      \begin{array}{ccc}
        0 & \alpha_2 & \beta_2' \\
        \alpha_1 & 0 & 0 \\
        \beta_1' & 0 & \gamma_1\gamma_2 \\
      \end{array}
    \right)=\\
    \left(
      \begin{array}{ccc}
        I_r & 0 & 0 \\
        0 & \alpha_1 & 0 \\
        0 & \beta_1' &  \gamma_1  \\
      \end{array}
    \right)\left(
      \begin{array}{ccc}
        0 & I_r & 0 \\
        I_{r'} & 0 & 0 \\
        0 & 0 & I_{n-(r+r')} \\
      \end{array}
    \right)\left(
      \begin{array}{ccc}
        I_{r'} & 0 & 0 \\
        0 & \alpha_2 & \beta_2' \\
        0 & 0 & \gamma_2 \\
      \end{array}
    \right)\in \hh\cdot g.
\end{multline*}
Note that $\diag(\alpha_1,I_{2(k-r')},\tilde\alpha_1) \in
Sp(J_{2k})$ and $\diag(\alpha_2,I_{2(k'-r)},\tilde\alpha_2) \in
Sp(J_{2k'})$ and therefore that
\[
    h=\diag(I_r,\alpha_1^{-1},I_{2(k-r')},\tilde\alpha_1^{-1}) \in H
    \text{ and }
h'=\diag(I_{r'},\alpha_2^{-1},I_{2(k'-r)},\tilde\alpha_2^{-1}) \in
\overline{H'}.
\]
Thus,
\[
h\left(
      \begin{array}{ccc}
        0 & \alpha_2 & \beta_2' \\
        \alpha_1 & 0 & 0 \\
        \beta_1' & 0 & \gamma_1\gamma_2 \\
      \end{array}
    \right)h'=\left(
      \begin{array}{ccc}
        0 & I_r & \beta_2 \\
        I_{r'} & 0 & 0 \\
        \beta_1 & 0 & \gamma \\
      \end{array}
    \right) \in \hh \cdot g
\]
for some $\gamma \in G_{n-(r+r')},\,\beta_1$ and $\beta_2.$ Now note
that
\begin{multline*}
    \left(
      \begin{array}{ccc}
        I_r & \beta_2\gamma^{-1}\beta_1 & -\beta_2\gamma^{-1} \\
        0 & I_{r'} & 0 \\
        0 & 0 & I_{n-(r+r')} \\
      \end{array}
    \right)\left(
      \begin{array}{ccc}
        0 & I_r & \beta_2 \\
        I_{r'} & 0 & 0 \\
        \beta_1 & 0 & \gamma \\
      \end{array}
    \right)\left(
      \begin{array}{ccc}
        I_{r'} & 0 & 0 \\
        0 & I_r & 0 \\
        -\gamma^{-1}\beta_1 & 0 & I_{n-(r+r')} \\
      \end{array}
    \right)\\=\left(
      \begin{array}{ccc}
        0 & I_r & 0 \\
        I_{r'} & 0 & 0 \\
        0 & 0 & \gamma \\
      \end{array}
    \right) \in \hh \cdot g.
\end{multline*}
\end{proof}
\begin{lemma}\label{lemma: gamma form prop p}
Let $\gamma \in G_{n-(r+r')}$ and let
\[
    g=\left(
      \begin{array}{ccc}
        0 & I_r & 0 \\
        I_{r'} & 0 & 0 \\
        0 & 0 & \gamma \\
      \end{array}
    \right)
\]
then $\prop(g,r,r').$
\end{lemma}
\begin{proof}
Recall that $r+r'>0.$ Let
\[
    \sigma_1=\left(
      \begin{array}{cc}
          & I_{2(k-r')} \\
         w_{r'} &  \\
      \end{array}
    \right) \text{ and }\sigma_2=\left(
      \begin{array}{cc}
          & I_{2(k'-r)} \\
        w_r &  \\
      \end{array}
    \right).
\]
For $x=\sigma_1^{-1}\gamma \sigma_2$ we have by the induction
hypothesis $\prop_{n-(r+r')}(x,r,r').$ Fix $y \in H_{r,2(k'-r)}$
such that either
\begin{equation}\label{eq: char not one}
x^{-1}yx \in \overline H_{r',2(k-r')}\text{ and }
\theta(y,x^{-1}yx)\ne 1
\end{equation}
or
\begin{equation}\label{eq: trans inv orb}
r=r',\, x^{-1} y \ {}^tx \in \overline H_{r',2(k-r')}\text{ and }
\theta(y,x^{-1}y\ {}^tx)= 1.
\end{equation}
For every invertible matrix $z$ denote by $z^\star$ the matrix $z$
if $y$ satisfies \eqref{eq: char not one} and the matrix ${}^t z$
otherwise. Note that if \eqref{eq: trans inv orb} holds then
$\sigma_1=\sigma_2$ and therefore in either case we have
\[
    x^\star=\sigma_1^{-1}\gamma^\star \sigma_2.
\]
There exist $s' \in Sp(J_{2(k-r')}),\,u' \in U_{r'}$ and $\varrho'
\in M_{r' \times 2(k-r')}(F)$ such that
\[
    \sigma_1 y\sigma_1^{-1}=
    \left(
      \begin{array}{cc}
          s' &  \\
        \varrho' &  {}^t(\tilde {u'})\\
      \end{array}
    \right)
\]
and there exist $s \in Sp(J_{2(k'-r)}),\,u \in U_{r}$ and $\varrho
\in M_{2(k'-r) \times r }(F)$ such that
\[
    \gamma^{-1}\sigma_1 y\sigma_1^{-1}\gamma^\star=
    \sigma_2x^{-1}yx^\star \sigma_2^{-1}=\left(
      \begin{array}{cc}
          s &  \varrho\\
         &  \tilde u\\
      \end{array}
    \right).
\]
Note then that
\begin{equation}\label{eq: value of char}
\theta(y,x^{-1}yx^\star)=\psi_r(u) \psi_{r'}(u')^{-1}.
\end{equation}
By \eqref{eq: symp property} there exist $(\varrho')^*\in M_{2(k-r')
\times r'}(F),\,\varrho^* \in M_{r \times 2(k'-r) }(F),z'$ and $z$
such that
\[
    h=\left(
  \begin{array}{ccc}
    {}^t u' & 0 & 0 \\
    (\varrho')^* & s' & 0 \\
    z' & \varrho' & {}^t \tilde u' \\
  \end{array}
\right)\in Sp(J_{2k}) \text{ and }h'=\left(
  \begin{array}{ccc}
    u & \varrho^* & z \\
    0 & s & \varrho \\
    0 & 0 &  \tilde u \\
  \end{array}
\right) \in Sp(J_{2k'}).
\]
Note that
\[
    g^\star=\left(
  \begin{array}{ccc}
    0 & I_r & 0 \\
    I_{r'} & 0 & 0 \\
    0 & 0 &  \gamma^\star \\
  \end{array}
\right).
\]
Let
\[
    \zeta_1=(\varrho^*,z)(\gamma^\star)^{-1} \text{ and }\zeta=(0_{r \times
    r'},\zeta_1)
\]
then
\[
    Y=\left(
  \begin{array}{cc}
    u &  \zeta \\
    0 & h
  \end{array}
\right) \in H,\,g^{-1} Yg^\star=\left(
  \begin{array}{cc}
    {}^t u' &  0 \\
    \zeta' & h'
  \end{array}
\right) \in \overline{H'}\text{ where
}\zeta'=\gamma^{-1}\left(\begin{array}{c} (\varrho')^* \\
z'\end{array}\right)
\]
and $\theta(Y,g^{-1} Yg^\star)=\psi_r(u)\psi_{r'}(u')^{-1}.$ The
property $\prop_n(g,r,r')$ therefore follows from \eqref{eq: value
of char} and the fact that either \eqref{eq: char not one} holds or
\eqref{eq: trans inv orb} holds.
\end{proof}

\subsection{Conclusion}
For $g \in G$, by \eqref{eq: bruhat decomp} there exists $w \in
{}_MW_{M'}$ such that $g \in Pw\overline{P'}.$ If $I_w$ is not empty
then $\prop(g,r,r')$ is proved in Lemma \ref{lemma: induction step}.
If $I_w$ is empty then we separated in \S\ref{sec: closed cell} the
statement $\prop(g,r,r')$ into 2 cases. If $g$ belongs to the
totally isotropic case then $\prop(g,r,r')$ follows from Lemma
\ref{lemma: reduce to representative+transpose}, Lemma \ref{lemma:
good representative} and Lemma \ref{lemma: gamma form prop p}.
Otherwise $\prop(g,r,r')$ follows from Lemma \ref{lemma: q equiv}
and Lemma \ref{lemma: not totally isotropic}. It follows that for
every $g \in G$ we have $\prop(g,r,r').$ Proposition \ref{prop: main
double cosets} now follows from Lemma \ref{lemma: reformulate prop}.
Therefore, Proposition \ref{prop: dist reduction} follows from
\S\ref{subsec: prop2 to prop1} and Theorem \ref{thm:uniq-disj}
follows from \S\ref{subsec: prop to thm}.

\end{document}